\documentclass[10pt]{amsart}
\usepackage[british,UKenglish,USenglish,american]{babel}
\usepackage[utf8]{inputenc}
\usepackage{braket}
\usepackage{xfrac}
\usepackage{amsfonts, amssymb, amsmath, amsthm, amscd}
\usepackage{mathtools}
\usepackage{tikz}
\usepackage{empheq}
\usepackage[left=3.8cm,right=3.8cm,top=3.8cm,bottom=3.8cm]{geometry}
\usepackage{color}
\usepackage{epsfig}  	
\usepackage{epic,eepic}   
\usepackage{setspace}
\usepackage{caption}
\DeclarePairedDelimiterX{\norm}[1]{\lVert}{\rVert}{#1}
\newtheorem{thm}{Theorem}[section]

\newtheorem{cor}[thm]{Corollary}

\theoremstyle{definition}

\theoremstyle{remark}
\newtheorem{remark}[thm]{Remark}

\DeclareMathOperator{\Vol}{Vol}

\DeclareMathOperator{\Area}{Area}
\DeclareMathOperator{\Length}{Length}
\DeclareMathOperator{\Ric}{Ric}
\DeclareMathOperator{\sys}{sys}
\DeclareMathOperator{\Hess}{Hess}

\begin{document}


\begin{abstract}
We prove some sharp systolic inequalities for compact $3$-manifolds with boundary. They relate the (relative) homological systoles of the manifold to its scalar curvature and mean curvature of the boundary. In the equality case, the universal cover of the manifold is isometric to a cylinder over a disc of  nonnegative constant curvature. 
\end{abstract}

\title{Sharp Systolic Inequalities for $3$-Manifolds with Boundary}

\author{Eduardo Longa}

\address{Departamento de Matem\'{a}tica, Instituto de Matem\'{a}tica e Estat\'{i}stica, Universidade de S\~{a}o Paulo, R. do Mat\~{a}o 1010,
S\~{a}o Paulo, SP 05508-900, Brazil}
\email{edulonga@ime.usp.br}



\maketitle

\let\thefootnote\relax\footnote{The author was partially supported by grant 2017/22704-0, São Paulo Research Foundation (FAPESP).}

\section{Introduction}

Systolic Geometry dates back to the late 1940s, with the work of Charles Loewner and his doctoral student Pao Ming Pu. This branch of differential geometry received more attention after the seminal work of Gromov \cite{gromov1983}, where he proved his famous systolic inequality and introduced many concepts, notably the filling radius and the filling volume of a manifold. This line of research would be popularised, subsequently, by Marcel Berger, in a series of books and articles (see \cite{berger1993}, \cite{bergerBook}, \cite{berger2008}, for example).

The main objects of Systolic Geometry are, as one would expect, the systoles. Let us recall the definition. For a Riemannian manifold $(M^n,g)$ and an integer $1 \leq k < n$, its \textit{homological $k$-systole} is given by
\begin{align*}
\sys_k(M) = \inf \{ \Vol(\Sigma) : \Sigma^k \subset M \text{ closed and embedded, } [\Sigma] \neq 0 \in H_k(M;\mathbb{Z}) \}. 
\end{align*} 

In a recent paper, Stern \cite{stern2019} gave a proof for the following systolic inequality, which is a weaker variant of the theorem proved by Bray-Brendle-Neves in \cite{bray2010}:

\begin{thm} \label{stern systolic}
On a closed, connected and oriented Riemannian $3$-manifold $(M^3, g)$ with positive scalar curvature $R_M$ and $H_2(M;\mathbb{Z}) \neq 0$, we have
\begin{align*}
\sys_2(M) \inf_M R_M \leq 8 \pi.
\end{align*} 
Moreover, if equality holds, then the universal cover of $M$ is isometric to the standard cylinder $\mathbb{S}^2 \times \mathbb{R}$ up to scaling, where $\mathbb{S}^2$ is the unit round sphere.
\end{thm}

From now on, we consider Riemannian $3$-manifolds with nonempty boundary. Let $(M^3,g)$ be such a manifold. Define its  \textit{relative homological $2$-systole} by 
\begin{align*}
\sys_2(M,\partial M) = \inf \{ \Area(\Sigma) : \Sigma \in \mathcal{S} \text{ and } [\Sigma] \neq 0 \in H_2(M, \partial M;\mathbb{Z}) \},
\end{align*}
where $\mathcal{S}$ denotes the set of all compact and embedded surfaces $\Sigma \subset M$ with boundary such that $\partial \Sigma \subset \partial M$. Notice that any surface in $\mathcal{S}$ that represents a nonzero element of $H_2(M, \partial M; \mathbb{Z})$ must be orientable.

\begin{remark}
Are there geometric obstructions for $H_2(M, \partial M; \mathbb{Z})$ to be nonzero for a given compact, connected and orientable Riemannian $3$-manifold $(M, g)$ with nonempty boundary? For instance, Fraser and Li  
\cite[Lemma 2.1]{fraser_li2014} showed that $H_2(M, \partial  M;\mathbb{Z})$ vanishes if the Ricci curvature of $M$ is nonnegative and $\partial M$ is strictly convex. They actually proved a much stronger result: under these geometric assumptions, $M$ is diffeomorphic to the $3$-ball $\mathbb{B}^3$ (see Theorem 2.11 in \cite{fraser_li2014}).
\end{remark}

In the present paper, we obtained sharp inequalities relating the (relative) homological systoles of a Riemannian $3$-manifold $(M^3, g)$ to its scalar curvature and mean curvature of the boundary. In all the results that follow, we assume that the scalar curvature of $M$ is positive (or nonnegative) and that $\partial M$ is mean-convex (or weakly mean-convex). Having this in mind, it is important to further elucidate the setting in which our results apply.

Recently, Carlotto and Li proved a complete topological classification of those compact, connected and orientable $3$-manifolds with boundary which support Riemannian metrics of positive scalar curvature and mean-convex boundary (see Theorem 1.1 and Corollary 2.6 in \cite{carlotto2018}). Namely, if $M^3$ is such a manifold, then there exist integers $A, B, C, D \geq 0$ such that $M$ is diffeomorphic to a connected sum of the form
\begin{align} \label{topology}
P_{\gamma_1} \# \cdots \# P_{\gamma_A} \# \mathbb{S}^3/ {\Gamma_1} \# \cdots \# \mathbb{S}^3 / {\Gamma_B} \# \left(  \#_{i=1}^C \mathbb{S}^2 \times \mathbb{S}^1 \right) \setminus \left( \sqcup_{i=1}^D B_i^3 \right),
\end{align}
where $P_{\gamma_i}$, $i \leq A$, are genus $\gamma_i$ handlebodies; $\Gamma_i$, $i \leq B$, are finite subgroups of $SO(4)$ acting freely on $\mathbb{S}^3$; $B_i^3$, $i \leq D$, are disjoint $3$-balls in the interior. Conversely, if $M$ is of this form, then it supports Riemannian metrics of positive scalar curvature and mean-convex boundary.

We now state our results. Firstly, inspired by Stern's ideas, we prove the following: 

\begin{thm} \label{main1}
Let $(M^3, g)$ be a compact, connected and oriented Riemannian $3$-manifold with nonempty boundary. Assume that $H_2(M, \partial M;\mathbb{Z}) \neq 0$. If $M$ has positive scalar curvature ($R_M > 0$) and weakly mean-convex boundary ($H^{\partial M} \geq 0$), then
\begin{align*}
\sys_2(M, \partial M) \inf_M R_M \leq 4 \pi.
\end{align*}
Moreover, if equality holds, then the universal cover of $M$ is isometric to the cylinder $\mathbb{S}^2_{+} \times \mathbb{R}$ up to scaling, where 
$\mathbb{S}^2_{+}$ is a closed hemisphere of the unit round sphere.
\end{thm}

A natural question in whether the number $\sys_2(M, \partial M)$ is attained by the area of a compact and properly embedded surface $\Sigma \subset M$. To answer this question, we first introduce some terminology. Recall that an embedded surface $\Sigma \subset M$ is said to be \textit{separating} if $M \setminus \Sigma$ has at least two connected components. We say that a compact and orientable $3$-manifold $M$ is \textit{weakly irreducible} if every smoothly embedded $2$-sphere in the interior of $M$ is separating. 

We then answer affirmatively the above question assuming some hypothesis on the geometry and topology of  $M$:

\begin{thm} \label{main2}
Let $(M^3,g)$ be a compact, connected and oriented Riemannian $3$-manifold with nonempty boundary. Assume that $M$ is weakly irreducible and $H_2(M, \partial M;\mathbb{Z}) \neq 0$. If $M$ has positive scalar curvature and weakly mean-convex boundary, then there exists a properly embedded free boundary stable minimal disc $D$ in $M$ such that $[D] \neq 0$ in $H_2(M, \partial M;\mathbb{Z})$ and $\Area(D) = \sys_2(M, \partial M)$.
\end{thm}

Using the same techniques as in Theorem \ref{main1}, we also prove:

\begin{thm} \label{main3}
Let $(M^3,g)$ be a compact, connected and oriented Riemannian $3$-manifold with nonempty boundary. Assume that 
$H_2(M, \partial M; \mathbb{Z}) \neq 0$ and that the connecting homomorphism $H_2(M, \partial M; \mathbb{Z}) \to H_1(\partial M; \mathbb{Z})$ is injective. If $M$ has positive scalar curvature and weakly mean-convex boundary, then
\begin{align*}
\frac{1}{2} \sys_2(M, \partial M) \inf_M R_M + \sys_1(\partial M) \inf_{\partial M} H^{\partial M} \leq 2 \pi.
\end{align*}
Moreover, if equality holds, then the universal cover of $M$ is isometric to the cylinder $\mathbb{B}^2_{r} \times \mathbb{R}$ up to scaling, where 
$\mathbb{B}^2_{r}$ is a geodesic ball of radius $r = \cos^{-1}\left( 1 - \frac{\sys_2(M, \partial M)}{2 \pi} \right)$ of the unit round sphere.
\end{thm}

It is worth mentioning that, just as Theorem \ref{stern systolic} is similar to Theorem 1 in \cite{bray2010}, Theorem \ref{main3} resembles the rigidity results of \cite{ambrozio}. Inarguably, the geometric invariants in \cite[Theorem 8]{ambrozio} and in Theorem \ref{main3} have similar expressions. Nonetheless, whereas we minimise area over all compact surfaces with boundary which are nontrivial in (relative) homology in order to define $\sys_2(M, \partial M)$ , L. Ambrozio minimises area over the set of all discs whose boundaries are curves in $\partial M$ that are homotopically nontrivial in $\partial M$ in order to define $\mathcal{A}(M,g)$. This may produce different results, as Fig. \ref{bitorus} shows.  
\begin{figure}[h] 
\centering
\includegraphics[width = 8cm]{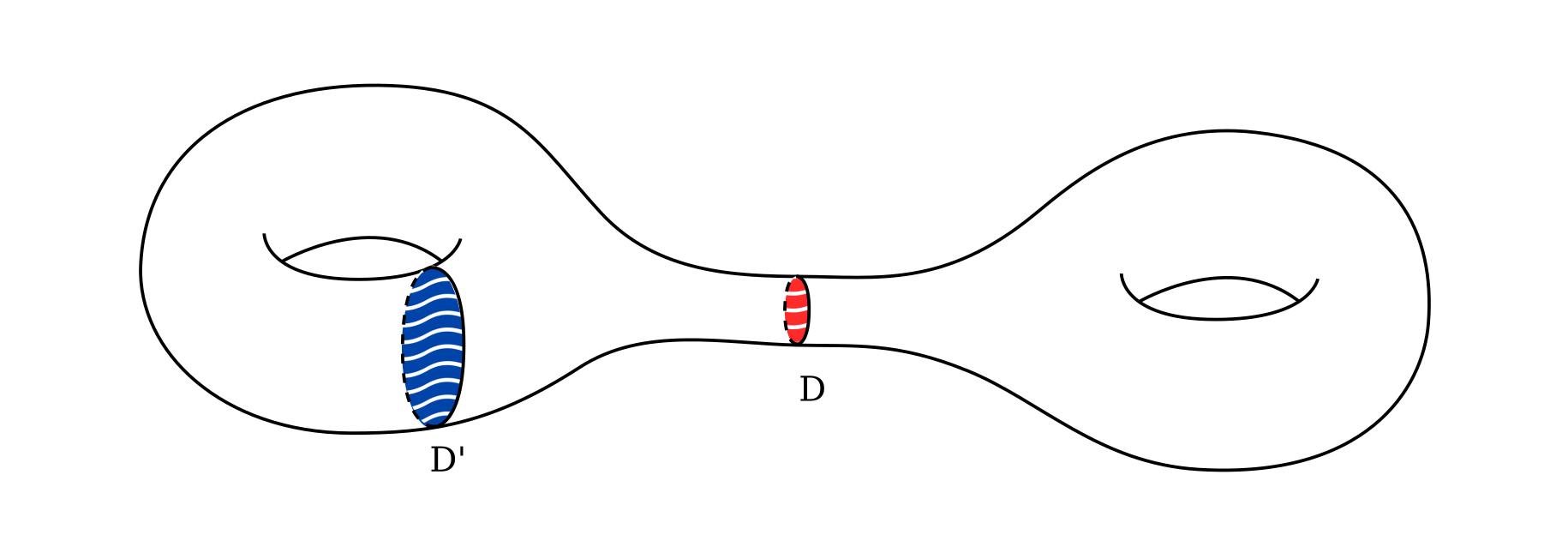}
\caption{The disc $D$ realises the quantity $\mathcal{A}(M,g)$, whereas $D'$ realises $\sys_2(M, \partial M)$. }
\label{bitorus}
\end{figure}

\begin{remark} If $M$ is compact, connected and orientable with connected nonempty boundary, then  the connecting homomorphism $H_2(M, \partial M; \mathbb{Z}) \to H_1(\partial M; \mathbb{Z})$ is injective if and only if  $H_2(M; \mathbb{Z}) = 0$. This can be seen by analysing the exact sequence in homology of the pair $(M, \partial M)$. 
\end{remark}

Finally, making a slight adaptation on the proof of Theorem \ref{main3}, we obtain the following corollary:

\begin{cor}
Let $(M^3,g)$ be a compact, connected and oriented Riemannian $3$-manifold with nonempty boundary. Assume that 
$H_2(M, \partial M; \mathbb{Z}) \neq 0$ and that the connecting homomorphism $H_2(M, \partial M; \mathbb{Z}) \to H_1(\partial M; \mathbb{Z})$ is injective. If $M$ has nonnegative scalar curvature and  mean-convex boundary then
\begin{align*}
\sys_1(\partial M) \inf_{\partial M} H^{\partial M} \leq 2 \pi.
\end{align*}
Moreover, if equality holds, then the universal cover of $M$ is isometric to the right circular cylinder $\mathbb{D}^2 \times \mathbb{R}$ up to scaling, where $\mathbb{D}^2$ is the flat closed unit disc.
\end{cor}

\section*{Acknowledgements}

The author would like to thank Paolo Piccione for the support during the period when this article was written and revised, Lucas Ambrozio for valuable comments on this work, Joshua Howie for providing the argument in Remark \ref{joshua} and Izabella Freitas for Fig.\ref{bitorus}. He would also like to thank the anonymous referee for pointing out the existence of Carlotto and Li's classification theorem stated in the introduction and for his/her interest on making the manuscript clearer and more precise.

\section{Proof of the Theorems}

For Theorems \ref{main1} and \ref{main3}, we shall make use of the following result:

\begin{thm}\smash{\cite[Theorem 1.1]{bray2019}} \label{stern theorem}
 Let $(M^3,g)$ be a compact, connected and oriented Riemannian $3$-manifold with nonempty boundary. For a harmonic map $u \colon M \to \mathbb{S}^1 = \mathbb{R} / \mathbb{Z}$ satisfying homogeneous Neumann condition, we have the inequality
\begin{align*}
2 \pi \int_{\mathbb{S}^1} \chi(\Sigma_{\theta}) \geq 
\int_{\mathbb{S}^1} \left( \int_{\Sigma_{\theta}} \frac{1}{2}(\norm{\mathrm{d}u}^{-2} \norm{\Hess(u)}^2 + R_M) 
+ \int_{\partial \Sigma_{\theta}} H^{\partial M} \right),
\end{align*}
where $R_M$ is the scalar curvature of $M$, $H^{\partial M}$ is the mean curvature of $\partial M$, $\Sigma_{\theta} = u^{-1}(\theta)$ is a regular level set of $u$ and $\chi(\cdot)$ denotes the Euler characteristic.
\end{thm}

Notice that every regular level set $\Sigma_{\theta}$ of $u$ in Theorem \ref{stern theorem} meets the boundary of $M$ orthogonally. This will be important in the proof of Theorem \ref{main1} below.

\begin{proof}[Proof of Theorem \ref{main1}]
Recall that Poincar\'{e}--Lefschetz duality gives an isomorphism
\begin{align*}
H_2(M, \partial M; \mathbb{Z}) \cong H^1(M;\mathbb{Z}) \cong [M : \mathbb{S}^1].
\end{align*}
Since we are assuming that $H_2(M, \partial M;\mathbb{Z}) \neq 0$, there is a nontrivial homotopy class $[v] \in [M : \mathbb{S}^1]$. Applying standard Hodge theory to the cohomology class $v^*(\mathrm{d} \theta) \in H^1_{\operatorname{dR}}(M)$ provides an energy-minimising representative $u \colon M \to \mathbb{S}^1$. It can be shown that this function is harmonic and satisfies homogeneous Neumann condition along 
$\partial M$.  

In order to prove the theorem, fix such a map $u \colon M \to \mathbb{S}^1$. By Theorem \ref{stern theorem} and by the fact that 
$H^{\partial M} \geq 0$, we have the following inequalities:
\begin{align}
2 \pi \int_{\mathbb{S}^1} \chi(\Sigma_{\theta}) &\geq 
\int_{\mathbb{S}^1} \left( \int_{\Sigma_{\theta}} \frac{1}{2}(\norm{\mathrm{d}u}^{-2} \norm{\Hess(u)}^2 + R_M) 
+ \int_{\partial \Sigma_{\theta}} H^{\partial M} \right) \nonumber \\
&\geq \frac{1}{2} \inf_M R_M \int_{\mathbb{S}^1} \Area(\Sigma_{\theta}). \label{ineq chi area}
\end{align}

Now notice that, whenever $\Sigma_{\theta}$ is a regular level set of $u$, it holds that every connected component of $\Sigma_{\theta}$ represents a nontrivial class in $H_2(M, \partial M;\mathbb{Z})$. Indeed, if $S$ is a connected component of $\Sigma_{\theta}$ and $h = u^*(\mathrm{d} \theta)$ is the gradient $1$-form induced by $u$, then
\begin{align*}
\int_S \ast h = \int_S \left\vert h \right\vert > 0,
\end{align*}
where $\ast h$ is the Hodge dual of $h$.

Also observe that, if $N(\theta)$ denotes the number of connected components of $\Sigma_{\theta}$, then $\chi(\Sigma_{\theta}) \leq N(\theta)$. This holds simply because $\chi(S) \leq 1$ for any compact and connected surface with boundary.

Combining these facts with inequality (\ref{ineq chi area}), we obtain

\begin{align*}
2 \pi \int_{\mathbb{S}^1} N(\theta) &\geq 2 \pi \int_{\mathbb{S}^1} \chi(\Sigma_{\theta}) 
\geq \frac{1}{2} \inf_M R_M \int_{\mathbb{S}^1} \Area(\Sigma_{\theta}) \\
& \geq \frac{1}{2} \sys_2(M, \partial M) \inf_M R_M \int_{\mathbb{S}^1} N(\theta). 
\end{align*}

Cancelling factors, we get
\begin{align*}
\sys_2(M, \partial M) \inf_M R_M \leq 4 \pi,
\end{align*}
as we wanted.

Suppose now that equality holds. Then, analysing all the previous steps, we have

\begin{itemize}
\item[(i)] $\Hess(u) \equiv 0$ on $M$;
\item[(ii)] $R_M \equiv \inf_M R_M > 0$  is constant along $M$;
\item[(iii)] $H^{\partial M} \equiv 0$ along $\partial M$;
\item[(iv)] $\chi(\Sigma_{\theta}) = N(\theta)$ for every $\theta \in \mathbb{S}^1$.
\end{itemize}

Firstly, notice that condition (i) implies that $\mathrm{d}u$ has constant norm (different from $0$). So, every level set $\Sigma_{\theta}$ is regular and totally geodesic. Indeed, let $A$ denote the second fundamental form of a level set of $u$, and let $X,Y$ be tangent vectors of that level set. Then
\begin{align*}
A(X,Y) &= \left\langle \overline{\nabla}_X Y, \frac{\nabla u}{\norm{\nabla u}} \right\rangle 
= \frac{1}{\norm{\nabla u}} \left\langle \overline{\nabla}_X Y, \nabla u \right\rangle
= -\frac{1}{\norm{\nabla u}} \left\langle \overline{\nabla}_X \nabla u, Y \right\rangle \\
&= -\frac{1}{\norm{\nabla u}} \Hess(u)(X,Y) 
=0.
\end{align*}

Secondly, the Bochner formula for the (harmonic) gradient $1$-form $h = u^*(\mathrm{d} \theta)$ reads
\begin{align*}
\Delta \frac{1}{2} \norm{h}^2 = \norm{Dh}^2 + \Ric(h,h).
\end{align*}
Since $\norm{h} = \norm{\mathrm{d}u}$ is constant and $Dh = \Hess(u) \equiv 0$, we get $\Ric(h,h) = \Ric(\nabla u, \nabla u) = 0$. Now, the Gauss 
equation for a level set $\Sigma_{\theta}$, 
\begin{align*}
\Ric(N,N) = \frac{1}{2}\left( R_M - R_{\theta} + H_{\theta}^2 - \norm{A_{\theta}}^2  \right),
\end{align*}
gives that the sectional curvature of $\Sigma_{\theta}$ is constant and equal to $\frac{1}{2} R_M$ (which is itself constant by (ii)). Here, 
$N = \frac{\nabla u}{\norm{\nabla u}}$ denotes the unit normal, $R_\theta$ the scalar curvature, $H_\theta$ the mean curvature and $A_\theta$ the second fundamental form of $\Sigma_\theta$. This way, each component of a level set of $u$ is isometric to disc (by condition (iv)) of a round sphere. 

We now show that the geodesic curvature of the boundary of such a disc $D$ is zero. For this, let $T$ be a unit vector field which is tangent to the boundary of $D$. Since $u$ satisfies the homogeneous Neumann boundary condition, $\{T(p), N(p) \}$ is an orthonormal basis of $T_p (\partial M)$ for every $p \in \partial D$, where $N = \frac{\nabla u}{\norm{\nabla u}}$. Now let $X$ be the outward unit  normal to $\partial M$. Then the geodesic curvature of $\partial D$ is given by $k_g = \langle \nabla_T X, T \rangle$, because $X$ is also the outward unit conormal to $\partial D$ in $D$. So, we have:
\begin{align*}
0 = H^{\partial M} =k_g + \langle \nabla_N X, N \rangle.
\end{align*}
But as $\langle X, \nabla u\rangle = 0$, we can write the second term above as
\begin{align*}
\langle \nabla_N X, N \rangle &= \frac{1}{\norm{\nabla u}^2} \langle \nabla_{\nabla u} X, \nabla u \rangle = -\frac{1}{\norm{\nabla u}^2} \langle \nabla_{\nabla u} \nabla u, X \rangle \\
&= -\frac{1}{\norm{\nabla u}^2} \Hess(u)(\nabla u, X)
\end{align*}
Since we are assuming that the Hessian of $u$ vanishes, this shows that $k_g = 0$, as we wanted.

Finally, fixing a connected component $S$ of a level set of $u$, the gradient flow of $u$, $\Phi : S \times \mathbb{R} \to M$,
\begin{align*}
\frac{\partial \Phi}{\partial t} = \frac{\nabla u}{\norm{\nabla u}} \circ \Phi,
\end{align*}
defines a local isometry. Notice that since the gradient of $u$ is tangent to the boundary (by the Neumann condition), the flow is well defined. So, $\Phi$ is a covering map. This completes the proof of the theorem.
\end{proof}

We now prove Theorem \ref{main3}.

\begin{proof}[Proof of Theorem \ref{main3}]
Fix a nontrivial harmonic map $u \colon M \to \mathbb{S}^1$ satisfying homogeneous Neumann boundary condition. From the proof of the previous theorem, we know that any component $S$ of a regular level set $\Sigma_\theta$ of $u$ represents a nontrivial class in $H_2(M, \partial M; \mathbb{Z})$. Since the connecting homomorphism is injective, it follows that some component of $\partial S$ represents a nontrivial class in $H_1(\partial M; \mathbb{Z})$. So, invoking Theorem \ref{stern theorem} as before, we have
\begin{align*}
2 \pi \int_{\mathbb{S}^1} N(\theta) &\geq 2 \pi \int_{\mathbb{S}^1} \chi(\Sigma_{\theta}) \geq 
\int_{\mathbb{S}^1} \left( \int_{\Sigma_{\theta}} \frac{1}{2}(\norm{\mathrm{d}u}^{-2} \norm{\Hess(u)}^2 + R_M) 
+ \int_{\partial \Sigma_{\theta}} H^{\partial M} \right) \\
&\geq \frac{1}{2} \inf_M R_M \int_{\mathbb{S}^1} \Area(\Sigma_\theta) + \inf_{\partial M} H^{\partial M} \int_{\mathbb{S}^1} \Length(\partial \Sigma_\theta) \\
&\geq  \frac{1}{2} \sys_2(M, \partial M; \mathbb{Z}) \inf_M R_M \int_{\mathbb{S}^1} N(\theta) + \sys_1(\partial M) \inf_{\partial M} H^{\partial M} \int_{\mathbb{S}^1} N(\theta)
\end{align*}
Cancelling factors, we get
\begin{align*}
\frac{1}{2} \sys_2(M, \partial M) \inf_M R_M + \sys_1(\partial M) \inf_{\partial M} H^{\partial M} \leq 2 \pi,
\end{align*}
as we wanted. The analysis of the equality case goes as in the proof of Theorem \ref{main1}. Just notice that, in the present situation, the boundary of $M$ need not be minimal. Therefore, each component of a level set of $u$ is isometric to a geodesic ball of area $\sys_2(M, \partial M)$ of the round sphere, as we wanted.
\end{proof}

We conclude by proving Theorem \ref{main2}.

\begin{proof}[Proof of Theorem \ref{main2}]
Let $\{ \Sigma_n \}_{n \geq 1}$ be a minimising sequence for $\sys_2(M, \partial M)$, that is, each $\Sigma_n$ is a compact, oriented and embedded surface determining a nonzero class in $H_2(M, \partial M;\mathbb{Z})$, and $\Area(\Sigma_n) \rightarrow \sys_2(M, \partial M)$ as $n \to \infty$. Each $\Sigma_n$ can be seen as an integer relative $2$-cycle. We then minimise the mass among all relative integral $2$-cycles in the homology class of $\Sigma_n$ (see Corollary 9.9 in \cite{federer1960}, for example). This gives rise to an integral relative cycle $\alpha_n$ whose support consists of a smooth, orientable, stable and properly embedded minimal surface $\Sigma_n^\ast$ (see the proof of Proposition 5.3 in \cite{guang}, the proof of Proposition 10 in \cite{mazet} and references therein). We are going to show that, under our assumption on the weak irreducibility of $M$, there is at least one component of  $\Sigma_n^*$ which has a nonempty boundary.

Assume $\Sigma_n^*$ has a closed component $S$. Since $S$ and $M$ are oriented, $S$ must be two-sided. Using $1$ as a test function for the second variation of area of this surface, and recalling that the scalar curvature of $M$ is positive, we conclude that this component must be a sphere (see the proof of Theorem 5.1 in \cite{schoen}). By hypothesis, this sphere is separating and thus determines the zero class in $H_2(M, \partial M;\mathbb{Z})$. So, if all the components of $\Sigma_n^*$ were closed, $\Sigma_n^*$ would represent the trivial class in $H_2(M, \partial M; \mathbb{Z})$, contradicting the fact that $\alpha_n$ is homologous to $\Sigma_n$ and $[\Sigma_n] \neq 0$.

Thus, from what we have just shown and since $\Sigma_n^\ast$ is not null-homologous (in relative homology), it must have at least one component with boundary which is also not null-homologous. Call it $D_n$.  Then $\{D_n\}_{n \geq 1}$ is a sequence of free boundary and properly embedded stable minimal surfaces of uniformly bounded area. By Theorem 1.2 and Theorem 6.1 in \cite{guang2018}, there is a subsequence $\{D_{n_k}\}_{k \geq 1}$ that converges smoothly and locally uniformly to a free boundary and properly embedded minimal surface $D$.  It is easy to see from the convergence that $[D] \neq 0$. Moreover, since $\Area(D_n) \leq \Area(\Sigma_n^\ast) \leq \Area(\Sigma_n)$, it follows that $\Area(D) = \sys_2(M, \partial M)$. In particular, $D$ minimises area in its homology class, which implies that $D$ is stable. Finally, we use Theorem 1.2 (iii) in \cite{fraser2014} to conclude that $D$ is a disc.
\end{proof}

\begin{remark} \label{joshua}
Having in mind Carlotto and Li's topological classification mentioned in the introduction for those $3$-manifolds admitting Riemannian metrics of positive scalar curvature and mean-convex boundary, the additional condition of being weakly irreducible in Theorem \ref{main2} holds if and only if there are no $\mathbb{S}^2 \times \mathbb{S}^1$ summands in (\ref{topology}) (that is, $C = 0$). Indeed, suppose that $C = 0$ and let $\Sigma$ be a smoothly embedded $2$-sphere in the interior of $M$. Let $\{S_j\}_{j=1}^k$ be a collection of embedded $2$-spheres which decompose $M$ into prime summands. Look at the intersection of $\Sigma$ with $\{S_j\}$ and let $\Delta$ be an innermost disc on some $S_k$. We can surger $\Sigma$ along $\Delta$, which will decompose $\Sigma$ into two embedded $2$-spheres. Repeating this process until there are no intersections with $\{S_j\}_{j=1}^k$, $\Sigma$ is decomposed into a collection of embedded $2$-spheres which we call $\Sigma'$. Then $\Sigma$ will be nonseparating in $M$ if and only if some component of $\Sigma'$ is nonseparating in $M$.
	
Each component of $\Sigma'$ is contained within a single prime summand of $M$. We then cut $M$ along $\{S_j\}_{j=1}^k$, and glue $3$-balls onto each $2$-sphere boundary component of the resulting $3$-manifolds. It is well-known that handlebodies and spherical $3$-manifolds are irreducible, which means that every embedded $2$-sphere bounds a $3$-ball. Thus each component of $\Sigma'$ is separating in its respective prime summand and hence in $M$. Removing the $D$ $3$-balls from $M$ which are disjoint from $\Sigma$ and $\Sigma'$ does not affect whether $\Sigma$ is separating. Therefore $\Sigma$ is separating in $M$.
	
Conversely, if $C\neq 0$, then we can find a nonseparating 2-sphere $\Sigma_0$ in some $\mathbb{S}^2\times \mathbb{S}^1$ component which is disjoint from each $B_i$ and each $S_j$. Furthermore, there exists a closed simple curve in $\mathbb{S}^2 \times \mathbb{S}^1$ which intersects $\Sigma_0$ transversely exactly once. Isotoping this curve to be disjoint from the balls $B_i$ and the spheres $S_j$ yields another curve, now in $M$, with the same property. Therefore $\Sigma_0$ is nonseparating in $M$.
\end{remark}

\end{document}